\documentclass[11pt]{amsart}

\usepackage{amssymb, graphicx}
\usepackage{graphicx}

\usepackage[all]{xy}

\usepackage{amsfonts}
\usepackage{amsmath}
\usepackage[foot]{amsaddr}
\usepackage{latexsym}
\usepackage{amscd}
\usepackage{verbatim}

\newcommand{\mf}{\mathfrak}
\newcommand{\g}{\mf{g}}
\newcommand{\h}{\mf{h}}

\allowdisplaybreaks[1]

\numberwithin{equation}{section}
\newtheorem{theorem}{Theorem}[section]
\newtheorem{proposition}[theorem]{Proposition}
\newtheorem{lemma}[theorem]{Lemma}
\newtheorem{corollary}[theorem]{Corollary}
 
\theoremstyle{remark}

\newcommand{\be}{\beta} 
\renewcommand{\a}{\alpha} 
\newcommand{\s}{\sigma}

\newcommand{\C}{\mathbb{C}}
\newcommand{\D}{\Delta}
\newcommand{\Dp}{\Delta^+}

\begin{document}
\title{On special covariants in the exterior algebra of a simple Lie algebra}

 \author{Corrado De Concini,
Pierluigi M\"oseneder Frajria, Paolo Papi}\author{Claudio Procesi}
\address[De Concini, Papi (corresponding author), Procesi]{Dipartimento di Matematica, Sapienza Universit\`a di Roma, P.le A. Moro 2,
00185, Roma, Italy; }
\address[M\"oseneder Frajria]{Politecnico di Milano, Polo regionale di Como, 
Via Valleggio 11, 22100 Como,
Italy.}
\email{deconcin@mat.uniroma1.it, pierluigi.moseneder@polimi.it} 
\email{papi@mat.uniroma1.it, procesi@mat.uniroma1.it}
  \keywords{Exterior algebra, invariants,  little adjoint representation, small representation.}
  \subjclass[2010]{17B20}

 \begin{abstract}  We study the subspace of the exterior algebra of a simple complex Lie algebra  linearly spanned by the copies of  the little adjoint representation or, in the case of the Lie algebra of traceless matrices, by the copies of the n-th symmetric power of the defining representation. As main result  we prove that this subspace is a  free module over the subalgebra of the exterior algebra generated by all primitive invariants except the one of highest degree.
 \end{abstract}
  \maketitle
 \section{Introduction}
  Let $\g$ be a simple Lie algebra (over $\C$) of rank $r$. In \cite{DPP} , the isotypic component $A=Hom_\g(\g,\bigwedge \g)$ of the adjoint representation in the exterior algebra of $\g$ has been studied. Recall that the invariant algebra $(\bigwedge \g)^\g$ is  
an exterior
algebra $\bigwedge(P_1,\dots,P_r)$ over primitive generators $P_i$ of degree $2m_i + 1$, where the integers $m_i$ (with $m_1\le\dots\le m_r$) are the exponents of $\g$. 
The main result of \cite{DPP} states that $A$ is a free algebra of rank $2r$ over the algebra $\bigwedge(P_1,\dots,P_{r-1})$. The purpose of this short paper is to single out other 
instances of this special behavior. We prove that the space $Hom_\g(L,\bigwedge \g)$ is a free algebra of dimension twice the dimension of the $0$-weight space of $L$ in the following remarkable cases:
\begin{enumerate}
\item $L$ is the 
{\it little adjoint} representation $L(\theta_s)$, i.e. the $\g$-module with highest weight the highest short root of $\g$;
\item  $\g$ is of type $A_{n-1}$ and $L=S^n(V)$ is the $n$-th symmetric power of the
defining representation $V$. Clearly, also its dual representation  shares  this property. 
\end{enumerate}
In order to build up  free generators in the little  adjoint case we are going to use a result of  Broer \cite {Broer}. Once we have the correct candidates, 
the proof of the statement will follow by slight modifications of the  machinery developed in \cite{DPP} for the adjoint representation. The case of $S^n(V)$ is dealt with by using classical invariant theory.\par
The adjoint  representation, the little adjoint representation, $S^n(V)$  and its dual are examples of small representations (see Section \ref{22}). For a small representation, and in fact only for a small representation,  one has  (see \cite {R})  that its multiplicity in $\bigwedge \g$ equals $2^r$ times the dimension of its zero weight space, a fact that we are going to use below.
It is natural to ask whether covariants of small modules have the nice behavior described above. It is easy to provide counterexamples. In this respect, using a result of Stembridge,
we are able to show that in type $A$ the adjoint representation, $S^n(V)$ and $S^n(V)^*$ are the only small modules whose covariants are freely generated over $\bigwedge(P_1,\dots,P_{r-1})$. Computer computations
show that no other example arises among small modules for any Lie algebra of rank at most 5.\par
The analysis of covariants of the little adjoint representation when $\g$ is of classical type has been also performed in \cite{Dolce} using classical invariant theory.
 \section{The little adjoint module}\label{22}
 Let $\g$ be a simple Lie algebra (over $\C$) of rank $r$. Fix a Cartan subalgebra
$\h$ in $\g$. Let $\D$ be the corresponding root system, $W$ the Weyl group, $\Dp$ a positive system and $\Pi$ the corresponding simple system. Let $\D_l,\D_s$ denote the sets of long and short roots, respectively; set also 
$\D^+_s=\D_s\cap\Dp,\,\D^+_l=\D_l\cap\Dp,\,\Pi_s=\D_s\cap\Pi,\,\Pi_l=\D_l\cap\Pi, r_s=|\Pi_s|,\,r_l=|\Pi_l|$.
Let $(\cdot,\cdot)$ denote the Killing form. If $\a\in \h^*$, we let $h_\a$ be the unique element of $\h$ such that $(h_\a,h)=\a(h)$ for all $h\in \h$. We use this form $(\cdot,\cdot)$ to identify $\g$ and $\g^*$
 when convenient. \par 
 Assume that $\g$ is not simply laced. Let $\theta_s$ be the highest (w.r.t. $\Dp$) short root. The irreducible $\g$-module $L(\theta_s)$ of highest weight $\theta_s$ is called the little adjoint representation. 
 
We are interested in the study of
$$LA:=Hom_\g(L(\theta_s),\bigwedge \g).$$
 $LA$ is the space of $\g$-equivariant maps from $L(\theta_s)$ to the space of multilinear alternating functions on $\g$. Clearly $LA$ is a (left or right) module over $(\bigwedge \g)^\g$. 

If $L$ is a $\g$-module, we  denote by $L_0$ its zero weight subspace. 
We shall prove the following:
\begin{theorem}\label{maint}
$LA$ is freely generated over $\bigwedge(P_1,\dots,P_{r-1})$  by  $2\dim L(\theta_s)_0$  generators, which can be explicitly described.
\end{theorem}
As an application, we can recover the following result of Bazlov \cite{B}.
\begin{corollary}\label{B}The Poincar\'e polynomial $GM_{\theta_s} (q)$,  describing the dimension of $LA$ in each degree, is given by
 \begin{equation}
\label{ps}GM_{\theta_s} (q)=(1+q^{-1}) \prod_{i=1}^{r-1}
(1+q^{2m_i+1})q^{m_r+1-2(r_s-1)r_l}\frac{1-q^{4r_lr_s}}{1-q^{4r_l}}.\end{equation} 
\end{corollary}
As explained in the Introduction, our main tools are a result of Broer on covariants of small modules and the machinery developed in   \cite{DPP} to prove Êthe analogue of Theorem \ref{maint}  for the adjoint representation.\par
Let us describe Broer's result. Recall that a finite dimensional $\g$-module $L$ is called {\it small}Ê if twice a root is not a weight of $L$.  \begin{theorem}\cite[Theorem1]{Broer}\label{broer} Let $L$ be a small $\g$-module. Then 
$Hom_\g(L,S(\g))$ is  isomorphic by restriction to 
$Hom_W(L_0, S(\h))$ as a module for $S(\h)^W\simeq S(\g)^\g$. \end{theorem}
It is easy to check (see \cite{R}) that  an highest weight module $L(\lambda)$ is small if and only if $\lambda\not\geq 2\eta$ for any dominant root $\eta$ of 
$\g$. In particular, both the adjoint and the little adjoint representations are small.
In order to apply Theorem \ref{broer}, we start with a simple observation. Let $H$ be the subgroup of $W$ generated by the reflections $s_\alpha$ with $\alpha$ long. 
\begin{lemma}\label{Htrivial}
$H$ acts trivially on $L(\theta_s)_0$.
\end{lemma}
\begin{proof}
The weights of $L(\theta_s)$ are precisely $\D_s\cup\{0\}$. It follows that, if $\alpha$ is a long root, and $e_\a,f_\a$ are root vectors in $\g_\a,\,\g_{-\a}$, respectively,  then $\exp(e_\alpha)\exp(-f_\alpha)\exp(e_\alpha)$ acts trivially on $L(\theta_s)_0$.
\end{proof}

Let $W_s$ be the reflection subgroup of $W$ generated by the reflections $s_\alpha$ with $\alpha \in \Pi_s$.

\begin{lemma} \
\begin{enumerate}
\item\label{1} $W= W_s\ltimes H$ so $W/H$ is canonically isomorphic to $W_s$.
\item\label{2} The isomorphism in \eqref{1} turns $L(\theta_s)_0$ into a $W_s$-module isomorphic to the reflection representation of $W_s$.\end{enumerate}
\end{lemma}
\begin{proof} The proof of 
\eqref{1} is given in \cite{Pan}, Proposition 2.1. We now prove \eqref{2}. By Lemma \ref{Htrivial}, $L(\theta_s)_0$ is a $W/H$-module. The isomorphism given in \eqref{1} is the one induced by the embedding of $W_s$ in $W$. To prove our claim we need to provide a bijective map $Span
(\Pi_s)\to L(\theta_s)_0$ and check that this map  intertwines the action of $W_s$. We realize $L(\theta_s)_0$ explicitly as follows: $\g$ is the fixed point subalgebra of a diagram automorphism $\s$ of a larger simple Lie algebra $\mf{a}$. Let $k$ be the order of $\s$ ($k=2$ or $3$). Let $\xi$ be a primitive $k$-th  root of unity. Then $L(\theta_s)$ is the $\xi$-eigenspace of  $\s$ in $\mf{a}$. Let $\h'$ be a $\s$-stable Cartan subalgebra of $\mf{a}$ containing $\h$. Since there is no root $\alpha$ of $\mf{a}$ such that $\alpha_{|\h}=0$, we have that $L(\theta_s)_0$ is the $\xi$-eigenspace of $\s_{|\h'}$. Let $\Pi'$ be the set of simple roots of $\mf{a}$ and $\Pi'_0$ the subset  of simple roots fixed by $\s$. Let $\Pi'_c$ be a connected component of $\Pi'\backslash\Pi'_0$. Then the map $\alpha\mapsto \alpha_{|\h}$ identifies $Span
(\Pi'_c)$ with $Span
(\Pi_s)$. Let $\pi$ be the orthogonal projection $\h'\to L(\theta_s)_0$. We define a map $Span(\Pi_s)\to L(\theta_s)_0$ by
$$\alpha_{|\h}\mapsto \pi(h_\alpha)=\frac{1}{k}\sum_{i=0}^{k-1}\xi^i\s^{k-i}(h_{\alpha}).$$
If $\alpha\in Span(\Pi'_c)$, then the $\s^i(\alpha)$ are pairwise orthogonal, hence the above formula implies that the map is injective. Since $\dim L(\theta_s)_0=\frac{1}{k-1}(rank(\mf{a})-rank(\g))$, it is easy to check that $\dim Span(\Pi_s)=\dim L(\theta_s)_0$. It follows that our map is bijective.

If $ \gamma\in \Pi_s$ then $\gamma=\a_{|\h}$ with $\a\in\Pi'_c$. 
Then $e_{\gamma}=\sum_{i=0}^{k-1}e_{\s^i(\alpha)}$ and  $f_{\gamma}=\sum_{i=0}^{k-1}f_{\s^i(\alpha)}$. Since the roots in the $\s$-orbit of $\alpha$ are orthogonal, we see that 
$$
s_\gamma=\exp(e_\gamma)\exp(-f_\gamma)\exp(e_\gamma)=\prod_{i=0}^{k-1}s_{\s^i(\alpha)},
$$
so the action of $s_{\gamma}$ on $\h'$ commutes with $\s$. It follows that, if $\be\in Span(\Pi_s)$ and $\be=\be'_{|\h}$ with $\be'\in\Pi'_c$, then
$$
s_\gamma(\be)\mapsto \pi(h_{s_\gamma(\be')})=s_\gamma(\pi(h_{\be'})).
$$
\end{proof}
By \cite{Pan}, $\mf{l}=\h\oplus \sum\limits_{\alpha\in\D_l}\g_\alpha$ is a semisimple equal rank subalgebra of $\g$ whose Weyl group is $H$. Obviously, the  action of $H$ on $\h$ is the reflection representation. 
Let $J_H=S(\h)^H$. Since $H$ is a normal subgroup of  $W$, it is clear that $J_H$ is $W$-stable.  

\begin{proposition}\label{JH}
 $\dim Hom_W(L(\theta_s)_0,J_H/J_H^2)=1$ and $\dim (J_H/J_H^2)^W=r_l$. 
\end{proposition}
\begin{proof}
The proof is a case by case check. In each case we will provide an explicit realization of the reflection representation of $W_s$ in a suitable $W$-stable space of basic invariants for $H$.
\vskip5pt
{\sl Type $C_r$.}  In this case  $\mf{l}$ is the product of $r$ copies of $A_1$,  so   $H=(\mathbb Z/2\mathbb Z)^r$. Let $\Dp_l=\{\beta_1,\dots,\beta_r\}$. Clearly $J_H/J_H^2\simeq Span(h^2_{\beta_1},\ldots,
h^2_{\beta_r})$.  It is easy to check that $W_s\simeq S_r$ and its action  on $\h$ is given by the permutation representation on the basis $\{h_{\beta_i}\}$. It follows that  $J_H/J_H^2$ is the sum of the reflection representation of $S_r$ and a 1-dimensional invariant space.
\vskip5pt
{\sl Type $B_r$.} In this case $\mf{l}$ is of type $D_r$, so $J_H=\C[p_0,p_1,\cdots,p_{r-1}]$, where $p_i$ is a basic invariant for $B_r$ of degree $2i$ if $i=1,2,\cdots,r-1$ and $p_0=\prod_{\alpha\in\D_s^+}h_\alpha$. Since $W_s$ has order $2$, generated by the reflection $s_\a$ w.r.t. the unique short simple root $\alpha$, we see that $\C p_0$ affords the reflection representation of $W_s$ and that $(J_H/J_H^2)^W\simeq Span(p_1,\ldots,p_{r-1})$.
\vskip5pt
{\sl Type $G_2$.} In this case $\mf{l}$ is of type $A_2$, so there are basic invariants $p_1$, $p_2$ for $H$ in degree $2$ and $3$ respectively. We can choose $p_1$ to be the basic invariant of degree $2$ for $W$. In this case $W_s\simeq \mathbb Z/2\mathbb Z$. Since $J_H\cap S^3(\h)=\C p_2$ we see that $\C p_2 $ is $W_s$-stable. Since $p_2$ is not $W$-invariant, we see that $W_s$ acts on $\C p_2$ by its reflection representation.
\vskip5pt
{\sl  Type $F_4$.} In this case $\mf{l}$ is of type $D_4$ and $W_s\simeq S_3$. Let $h_1,f_1,f_2, h_2$ be basic invariants for $H$ of degree $2,4,4,6$ respectively. The basic invariants for $W$ occur in degrees $2,6,8,12$. We can choose $h_1,h_2$ to be basic invariants for $W$. We claim that the action of $W_s$ on $Span(f_1,f_2)$ is given by its reflection representation. Indeed, since $Span(f_1,f_2)$ cannot contain invariants for $W_s$, the only other possibility is that $W_s$ acts on  $f_1,f_2$ by the sign representation. If this were the case, we would have that $\dim S^8(\h)^W\ge 5$. But we know that $\dim S^8(\h)^W=3$.
\end{proof}

\begin{proof}[Proof of Theorem \ref{maint}] Choose $q\in Hom_W(L(\theta_s)_0,J_H)$ so that $q$ induces the embedding of $L(\theta_s)_0$ in $J_H/J_H^2$ provided by Proposition \ref{JH}. We can choose $q$ to be homogeneous and we let $n_0$ be the degree of $q$. We can write 
$$
J_H/J_H^2=q(L(\theta_s)_0)\oplus (J_H/J_H^2)^W
$$
as a $W$-module.  Using the fact that $J_H=S(J_H/J_H^2)$ and Lemma \ref{Htrivial}, we can write
\begin{equation}\label{Hom}
Hom_W(L(\theta_s)_0,S(\h))=Hom_W(L(\theta_s)_0,J_H)=S((J_H/J_H^2)^W)\otimes Hom_{W_s}(L(\theta_s)_0,S(q(L(\theta_s)_0))).
\end{equation}
Since the action of $W_s$ on $L(\theta_s)_0$ is the reflection representation and $W_s$ is a reflection group of type $A$, it is known (see \cite{Kostant}) that $Hom_{W_s}(L(\theta_s)_0,S(L(\theta_s)_0))$ is freely generated over $S(L(\theta_s)_0)^{W_s}$ by $r_s$ homogeneous generators $g_1,\dots, g_{r_s}$ in degrees $1,2,\dots,r_s$. It follows from \eqref{Hom} that $q(g_i)$ ($i=1,\ldots, r_s$) are free generators for $Hom_W(L(\theta_s)_0,S(\h))$ over $S(\h)^W$ in degrees $n_0,2n_0,\ldots, r_sn_0$.

Theorem \ref{broer} now provides free generators $F_1,\ldots, F_{r_s}$ for $Hom_\g(L(\theta_s),S(\g))$ over $S(\g)^\g$ in degrees $n_0,2n_0,\ldots, r_sn_0$. 
Let $\delta:\bigwedge^i \g\to \bigwedge^{i+1} \g$ be the Koszul differential. Let $s:S(\g)\to\bigwedge \g$ be the map extending $\delta_{|\bigwedge^1\g}:\g\to\bigwedge^2\g$ to $S(\g)$. Since $s$ is a $\g$-equivariant map, composing with $s$ defines a map $Hom_\g(L,S(\g))\to Hom_\g(L,\bigwedge \g)$. Set $f_i=s\circ F_i$ and $u_i=\partial \circ f_i$. Here $\partial={}^t\delta$.\par We claim that $f_i,u_i$ are  free generators for $Hom_\g(L(\theta_s),\bigwedge \g)$ over $\bigwedge (P_1,\ldots, P_{r-1})$. From now on, we may proceed as in  \cite{DPP}.  Let us sketch the main steps.
By \cite[Corollary 4.2]{R}, we have that  $\dim LA=2^r\dim L(\theta_s)_0$,  hence it suffices to prove that $f_i,u_i$ are linearly independent over $\bigwedge (P_1,\ldots, P_{r-1})$. Writing a linear combination of $f_i,u_i$ with coefficients in $\bigwedge (P_1,\ldots, P_{r-1})$ and applying $\delta$ one readily reduces to prove that the $f_i$ are independent. 
Identify $Hom_\g(L(\theta_s),\bigwedge \g)$ with $(\bigwedge \g\otimes L(\theta_s))^\g$ and  fix a symmetric invariant bilinear form $\langle\cdot,\cdot\rangle$  on $L(\theta_s)$. For $a,b,\in \bigwedge \g,\,x,y\in L(\theta_s)$ we set 
$$e(a\otimes x,b\otimes y)=\langle x,y \rangle a\wedge b.$$ If instead $a,b\in S(\g)$, then we set
$$(a\otimes x,b\otimes y)=\langle x,y \rangle a b.$$\par
Now, as in \cite[Lemma 2.6]{DPP},
  the claim about the independence of the $f_i$ boils down to showing that 
\begin{equation}\label{f}
e(f_i,u_{r_s-i+1})=c_iP_r,\ c_i\ne 0. 
\end{equation}
Let $d:S(\g)\to S(\g)\otimes \g$ be the usual differential on functions and $m:\bigwedge\g\otimes \g\to \bigwedge\g$ the multiplication map. Define $t:S(\g)\to \bigwedge\g$ setting $m\circ (s\otimes 1)\circ d$.
The argument given in \cite{DPP} to prove formula (2.21) therein shows that, up to a nonzero constant,
$$
e(f_i,u_{r_s-i+1})=t((F_i,F_{r_s-i+1})).
$$
 Now observe that, by inspection, we have
\begin{equation}\label{m}
n_0=\begin{cases}\frac{m_r+1}{2}\quad&\text{if $r_s=1$,}\\
\frac{m_r+1}{2}-(r_s-1)r_l=2r_l\quad&\text{if $r_s>1$.}
\end{cases}
\end{equation}
This implies that $2n_0i+2n_0(r_s-i+1)-1=2n_0(r_s+1)-1=2m_r+1$. 

Recall that the range of the map $t$, when restricted to $S(\g)^\g$, is, by  e. g. \cite[Theorem 64]{K}, the space of primitive elements in $\bigwedge \g$,
so it is enough to check that $t((F_i,F_{r_s-i+1}))\ne 0$. This is equivalent to checking that, if $J^+$ is the ideal in $S(\g)^\g$ of elements of positive degree, then $(F_i,F_{r_s-i+1})\not\in (J^+)^2$. As in Lemma 2.8 in \cite{DPP}, we see that the restriction of $(F_i,F_{r_s-i+1})$ to $\h$ is $(q(g_i),q(g_{r_s-i+1}))$.  In  the proof of Proposition \ref{JH}, we identified $J_H/J_H^2$ with $\{(x_1,\ldots,x_{r_s+1})\in\C^{r_s+1}\mid \sum_ix_i=0\}$ in such a way that the action of $W_s$ on $J_H/J_H^2$  intertwines with the standard action of the symmetric group $S_{r_s+1}$ on the latter space. 
With this identification,  the generators $g_i$ can be chosen to correspond   precisely  to the differentials of     normalized Newton polynomials
$\displaystyle{\psi_{[i+1]}:=\frac1{i+1}\sum\limits_{k=1}^{r_s+1}x_k^{i+1}}$. We can conclude using the formula 
$$
(d\psi_{[k]},d\psi_{[g]})=  \sum_{i=1}^{r_s+1}x_i^{k+g-2}=(k+g-2)\psi_{[k+g-2]}
$$
(see \cite{DPP}).
\end{proof}
\begin{proof}[Proof of Corollary \ref{B}] The proof of   Theorem \ref{maint} shows  that 
$$GM_{\theta_s}(q)=(1+q^{-1})\prod_{i=1}^{r-1}
(1+q^{2m_i+1})q^{2n_0}(1+q^{2n_0}+\ldots+q^{2(r_s-1)n_0}).$$
Now formula \eqref{ps} follows from \eqref{m}.
\end{proof}  
\section{The module $S^n(V)$}\label{3}
In this section, $V$ is  a $n$-dimensional complex vector space and  $\mathfrak g=sl(V)$.   We sometimes assume to have chosen a trivialization $\bigwedge ^nV\simeq \mathbb{C}$, although for a formal step it is better not to think in this form.\smallskip

 We are interested in studying the isotypic component of  type  $S^n(V)$ (resp. $S^n(V^*)$) in $\bigwedge  \mathfrak g^*$, or the $\mathfrak g$-invariants of  $S^n(V^*)\otimes \bigwedge  \mathfrak g^*$,        (resp. $S^n(V )\otimes \bigwedge  \mathfrak g^*$). As we will see in the next Section,  $S^n(V)$ is a small representation, hence we can use   \cite[Corollary 4.2]{R} to deduce  that   
 \begin{equation}\label{reeder}
\dim((S^n(V^*)\otimes \bigwedge  \mathfrak g^*)^{\mathfrak g})=
\dim((S^n(V )\otimes \bigwedge  \mathfrak g^*)^{\mathfrak g})=2^{n-1}.
\end{equation}  

We think of  $\bigwedge^i  \mathfrak g^*$ as the space of multilinear alternating  functions in $i$ variables from $\mathfrak g$ to $\mathbb{C}$ and  of $\bigwedge^i  \mathfrak g^*\otimes  \bigwedge ^nV$ as the space of multilinear alternating  functions in $i$ variables from $\mathfrak g$ to $  \bigwedge ^nV$ (similarly for $  \bigwedge ^nV^*$).  

Recall that the primitive generators of the ring of invariants $( \bigwedge  \mathfrak g^*)^{\mathfrak g }$ are  the functions $T_i$  defined by 
$$T_i:=tr(St_{2i+1}( A_1 , A_2 ,\ldots, A_{2i }, A_{2i+1} )),$$
 where $St_n(x_1,\ldots,x_n)=\sum_{\sigma\in S_n}\epsilon_\sigma x_{\sigma(1)}\ldots x_{\sigma(n)}$ is the standard polynomial.

We  introduce equivariant maps  $$\Phi:S^n(V)\to \bigwedge ^n\mathfrak g^*\otimes  \bigwedge^nV,\quad \Psi:S^n(V)\to \bigwedge ^{n-1}\mathfrak g^*\otimes  \bigwedge ^nV$$ by assigning  homogeneous polynomial maps (cf. \cite[\S5, 2.3]{Procesi}) $$v\mapsto \Phi(v)\in \bigwedge ^n\mathfrak g^*\otimes  \bigwedge ^nV,\ v\mapsto \Psi(v)\in  \bigwedge ^{n-1}\mathfrak g^*\otimes  \bigwedge ^nV$$ 
defined, for $v\in V$,  as 
\begin{align*}\label{}
 &\Phi(v)(A_1,\ldots,A_n):=A_1v\wedge A_2v\wedge \ldots \wedge A_{n-1}v\wedge A_nv,\\  
 &\Psi(v)(A_1,\ldots,A_{n-1}):=A_1v\wedge A_2v\wedge \ldots\wedge  A_{n-1}v\wedge v.
\end{align*}
A similar formula  holds for  maps $\Phi^*:S^n(V^*)\to \bigwedge\limits^n\mathfrak g^*\otimes  \bigwedge\limits^nV^*,\quad \Psi^*:S^n(V^*)\to \bigwedge\limits^{n-1}\mathfrak g^*\otimes  \bigwedge\limits^nV^*$: when $\gamma\in V^*$ we set
\begin{align*}\label{}
 &\Phi^*(\gamma)(A_1,\ldots,A_n):=A_1^t \gamma\wedge A_2^t \gamma\wedge \ldots \wedge A_{n-1}^t\gamma\wedge A_n^t \gamma,\\
 &\Psi^*(\gamma)(A_1  ,\ldots,A_{n-1} ):=A_1^t \gamma\wedge A_2^t \gamma\wedge \ldots \wedge A_{n-1}^t\gamma\wedge \gamma.
\end{align*} We use the same symbols  $\Phi,\Psi$ to denote the corresponding elements in $(S^n(V^*)\otimes \bigwedge  \mathfrak g^*)^{\mathfrak g}$.  Notice that we have an equivariant pairing $ S^n(V )\times S^n(V^*)\to \mathbb{C}$, which gives, by duality, a canonical map $I:\C\to  S^n(V )\otimes S^n(V^*)$, and which induces an equivariant pairing $(\cdot,\cdot)$
\begin{equation*}\label{}
 Hom(S^n(V), \bigwedge  \mathfrak g^*\otimes  \bigwedge ^nV)\times 
 Hom(S^n(V^*), \bigwedge  \mathfrak g^*\otimes  \bigwedge ^nV^*)\to  \bigwedge  \mathfrak g^*\otimes  \bigwedge ^nV\otimes  \bigwedge ^nV^*= \bigwedge  \mathfrak g^*
 \end{equation*}  in the following way. We let $\langle \cdot|\cdot\rangle$ denote the natural pairing between 
 $V$ and $V^*$. We extend this pairing to define the canonical trivialization  $\bigwedge ^nV\otimes  \bigwedge ^nV^*\to \mathbb{C}$ by setting 
 \begin{equation}\label{catri}
 \langle v_1\wedge v_2\wedge\ldots\wedge v_n |\gamma_1\wedge \gamma_2\wedge\ldots\wedge \gamma_n\rangle= \det (\langle v_i\,|\,\gamma_j\rangle).
\end{equation}   The pairing $(a,b)$ is then defined by  computing in $1$ the composition
\begin{align*}\label{ccd0}
&\begin{CD}
\mathbb{C}@>I>>S^n(V)\otimes S^n(V^*)@>a\otimes b>>\bigwedge ^{i }\mathfrak g^*\otimes  \bigwedge ^nV\otimes \bigwedge ^{j}\mathfrak g^* \otimes  \bigwedge ^nV^*
\end{CD}\\
&\begin{CD}@>m>>\bigwedge ^{i +j}\mathfrak g^*\otimes  \bigwedge ^nV \otimes  \bigwedge ^nV^* @> \eta>>\bigwedge ^{i +j}\mathfrak g^*
\end{CD}.\notag
\end{align*}
Here  $m$ is exterior multiplication and the isomorphism $\eta$ is given by the canonical trivialization  \eqref{catri}.

Restricting to invariants we have   finally a pairing 
\begin{equation*}\label{} Hom(S^n(V^*), \bigwedge  \mathfrak g^*\otimes  \bigwedge ^nV^*)^{\mathfrak g}\times 
 Hom(S^n(V ), \bigwedge  \mathfrak g^*\otimes  \bigwedge ^nV )^{\mathfrak g}  \to (\bigwedge  \mathfrak g^*)^{\mathfrak g}.
\end{equation*}  
We want to compute $(\Psi,\Phi^*)$, so we want to understand the composed map
\begin{equation*}\label{}
 \begin{CD}
\mathbb{C}@>I>>S^n(V)\otimes S^n(V^*)@>\Psi\otimes \Phi^*>>\bigwedge \limits^{n-1 }\mathfrak g^*\otimes \bigwedge\limits^{n}V\otimes \bigwedge \limits^{n}\mathfrak g^* \otimes \bigwedge\limits ^{n }V^*@>\eta>>\bigwedge\limits ^{n-1 }\mathfrak g^* \otimes \bigwedge \limits^{n}\mathfrak g^*.
\end{CD}
\end{equation*} 
For this we can polarize, getting the following commutative diagram
\begin{equation*}\label{}
\begin{CD}
\mathbb{C}@>I>>S^n(V)\otimes S^n(V^*)@>\Psi\otimes \Phi^*>>\bigwedge\limits ^{n-1 }\mathfrak g^*\otimes \bigwedge\limits ^{n }V\otimes \bigwedge\limits ^{n}\mathfrak g^* \otimes \bigwedge \limits^{n }V^*@>\eta>>\bigwedge\limits ^{n-1 }\mathfrak g^*\otimes \bigwedge \limits^{n}\mathfrak g^* \\ @V1VV@VpVV@V1VV@ViVV\\
\mathbb{C}@>I>> V^{\otimes n}\otimes  (V^*)^{\otimes n}@>\psi\otimes \phi^*>>\bigwedge \limits^{n-1 }\mathfrak g^*\otimes \bigwedge\limits ^{n }V\otimes \bigwedge\limits ^{n}\mathfrak g^* \otimes \bigwedge \limits^{n }V^* @>\pi>> (\mathfrak g^{\otimes 2n-1})^*
\end{CD}.
\end{equation*} The map $p(v^n\otimes \gamma^n):= v^{\otimes n}\otimes \gamma^{\otimes n}$ is polarization, the map $i$  is the embedding of multilinear functions alternating in two blocks of variables into multilinear functions, the map $\pi$ is the (external) multiplication of multilinear functions composed with the canonical trivialization.
  The polarized maps  $\phi$ and $\psi$ are given by 
\begin{align*}\label{}
&\psi(v_1,\ldots,v_n)(A_1,\ldots,A_{n-1}):=(n!)^{-1}\sum_{\sigma\in S_n}A_1v_{\sigma(1)}\wedge A_2v_{\sigma(2)}\wedge \ldots \wedge A_{n-1}v_{\sigma(n-1)}\wedge  v_{\sigma(n)},
\\
&\phi^*(\gamma_1,\ldots,\gamma_n)(B_1,\ldots,B_{n}):=(n!)^{-1}\sum_{\tau\in S_n}B^t_1\gamma_{\tau(1)}\wedge B^t_2\gamma_{\tau(2)}\wedge \ldots \wedge B^t_{n-1}\gamma_{\tau(n-1)}\wedge   B^t_n\gamma_{\tau(n)},
\end{align*} 
thus
\begin{align*}\label{}&\pi\circ(\psi\otimes\phi^*)(v_1,\ldots,v_n,\gamma_1,\ldots,\gamma_n)(A_1,\ldots,A_{n-1},B_1,\ldots,B_{n})  
\\
\label{finale}&=\langle \psi(v_1,\ldots,v_n)(A_1,\ldots,A_{n-1})\,|\, \phi^*(\gamma_1,\ldots,\gamma_n)(B_1,\ldots,B_{n})\rangle =(n!)^{-2}\sum_{\sigma,\tau\in S_n}\langle A_\sigma v| B_\tau\gamma\rangle\end{align*}
where for shortness we have set 
$ A_\sigma v= A_1v_{\sigma(1)}\wedge A_2v_{\sigma(2)}\wedge \ldots \wedge A_{n-1}v_{\sigma(n-1)}\wedge  v_{\sigma(n)},$ and $ B_\tau\gamma=B^t_1\gamma_{\tau(1)}\wedge B^t_2\gamma_{\tau(2)}\wedge \ldots \wedge B^t_{n-1}\gamma_{\tau(n-1)}\wedge  B^t_n \gamma_{\tau(n)}.
$\par
We have also, setting $A_n=1_V$,
\begin{equation}\label{final}
 \langle A_\sigma v | B_\tau \gamma\rangle
=\sum_{\lambda\in S_n} \epsilon_\lambda\prod_{h=1}^n\langle A_hv_{\sigma(h)}\,|\,B^t_{\lambda(h)}\gamma_{\tau\circ\lambda(h)} \rangle
=\sum_{\lambda\in S_n} \epsilon_\lambda\prod_{h=1}^n\langle B_{\lambda(h)}A_hv_{\sigma(h)}\,|\, \gamma_{\tau\circ\lambda(h)} \rangle.
\end{equation} Consider $ \langle A_\sigma v | B_\tau \gamma\rangle$ as a function on $V^{\otimes n}\otimes  (V^*)^{\otimes n}=End(V)^{\otimes n}$. The image of the canonical element $I$ in $End(V)^{\otimes n}$ is $1_V^{\otimes n}$ and we want to compute $ \langle A_\sigma v | B_\tau \gamma\rangle$ on this canonical element.  

For this define formally  matrix variables  $Y_i=v_i\otimes\gamma_i$. We first compute  $\langle A_\sigma v | B_\tau \gamma\rangle$ on all elements $Y_1\otimes Y_2\otimes\ldots \otimes Y_n\in End(V)^{\otimes n}$; then we set all $Y_i=1_V$ in order to perform the desired computation. 

More in detail, we proceed as follows.
For $X_1,\ldots,X_n\in\g$, set
\begin{equation*}\label{final1} I_{\sigma,\tau}:=
\prod_{h=1}^n \langle X_iv_{\sigma(i)}\,|\, \gamma_{\tau(i)} \rangle =
\prod_{h=1}^n \langle X_{\sigma^{-1}(i)}v_{ i }\,|\, \gamma_{\tau\circ\sigma^{-1}(i)} \rangle. \end{equation*}

In order to explicit this formula set $w_i=X_{\sigma^{-1}(i)}v_i$ and $Z_i^{\sigma}=X_{\sigma^{-1}(i)}\circ Y_{ i }=X_{\sigma^{-1}(i)}v_{ i }\otimes\gamma_{ i }=w_i\otimes\gamma_{ i }$.  We have
\begin{equation*}\label{final2} I_{\sigma,\tau}=
\prod_{i=1}^n  \langle w_i\,|\, \gamma_{\tau\circ \sigma^{-1}(i)} \rangle .
\end{equation*}

\smallskip

Recall that, if we take matrix variables $W_i:=w_i\otimes\gamma_i$ and a permutation $\mu$,  then $\prod_i\langle w_i\,|\,\gamma_{\mu(i)}\rangle$  is the multilinear invariant of $n$ matrices  $\phi_\mu(W_1,\ldots,W_n):=\prod tr(M_j)$, where the monomials $M_j$ are the products of the $W_i$ over the indices $i$ appearing in the cycles of  $\mu$. 
It follows that we have the formula
\begin{equation}\label{conj}
\phi_\mu(W_{\tau(1)},\ldots,W_{\tau(n)})=\phi_{\tau\mu\tau^{-1}}(W_1,\ldots,W_n).\end{equation}
Clearly,\begin{equation}\label{}
I_{\sigma,\tau} =\phi_{\tau\circ\sigma^{-1}}(Z_1^{\sigma},Z_2^{\sigma},\ldots,Z_n^{\sigma})).
\end{equation} When we compute this invariant on the canonical element, this is equivalent to setting all $Y_i=1_V$,  hence $Z_i^{\sigma}=X_{\sigma^{-1}(i)}\circ Y_{ i }$ becomes $ X_{\sigma^{-1}(i)}$ and we get as evaluation $$\phi_{\tau\circ\sigma^{-1}}(X_{\sigma^{-1}(1)},X_{\sigma^{-1}(2)},\ldots,X_{\sigma^{-1}(n)}) =\phi_{\sigma^{-1}\circ \tau}(X_{ 1 },X_{ 2 },\ldots,X_{ n }).$$ 
In the last equality we have used \eqref{conj}.
Setting $X_i=B_{\lambda(i)}A_i$ we find  \begin{equation*}\label{}
 \langle A_\sigma v | B_\tau \gamma\rangle(I)=
\sum_{\lambda\in S_n} \epsilon_\lambda\phi_{\sigma^{-1}\circ\tau\circ\lambda}(B_{\lambda(1)}A_1 ,B_{\lambda(2)}A_2 ,\ldots,B_{\lambda(n-1)}A_{n-1},B_{\lambda(n)} ),
\end{equation*}
so  that
 \begin{equation}\label{supfi}
\pi\circ( \phi\otimes \psi^*)\circ I(1)=(n!)^{-2}\sum_{\sigma,\tau,\lambda}\epsilon_\lambda\phi_{\sigma^{-1}\circ\tau\circ\lambda}(B_{\lambda(1)}A_1 ,B_{\lambda(2)}A_2 ,\ldots,B_{\lambda(n-1)}A_{n-1},B_{\lambda(n)} ).
\end{equation} 
Recall that $(\Psi,\Phi^*)=(m\circ\eta\circ(\Psi\otimes \Phi^*)\circ I)(1)$.
For any vector space $U$ we identify   the space $\bigwedge \limits^kU^*$ with the subspace of $(U^*)^{\otimes k}$ formed by the alternating multilinear functions. Under this embedding, a  decomposable element $\phi_1\wedge\ldots\wedge\phi_k$ corresponds to the function
$$f(x_1,\ldots,x_k):=\sum_{\sigma\in S_k}\epsilon_\sigma \phi_1(x_{\sigma(1)})\ldots  \phi_1(x_{\sigma(k)}).$$
The alternator operator on $(U^*)^{\otimes n}$ is 
$$Alt:y_1\otimes\ldots\otimes y_n\mapsto \frac{1}{n!}
\sum_{\sigma\in S_n}\epsilon_\sigma y_{\sigma(1)}\otimes\ldots\otimes  y_{\sigma(k)} .$$

The relation between exterior multiplication of alternating functions and of multilinear functions  is given by the following commutative diagram
$$
\begin{CD}
\bigwedge\limits^{h}U^*\otimes\bigwedge\limits^k U^*@>m>> \bigwedge \limits^{h+k}U^* \\ @ViVV@V\frac{1}{\binom{h+k}{h}}iVV\\
(U^*)^{\otimes h+k}@>Alt>> (U^*)^{\otimes h+k}
\end{CD}
$$
which in our setting reads
$$
\begin{CD}
\bigwedge\limits^{n-1}\g^*\otimes\bigwedge\limits^n\g^*@>m>> \bigwedge \limits^{2n-1}\mathfrak g^* \\ @ViVV@VC_n1VV\\
(\mathfrak g^*)^{\otimes 2n-1}@>Alt>> \bigwedge \limits^{2n-1}\mathfrak g^*
\end{CD},
$$
where $C_n=\tfrac{(n-1)!n!}{(2n-1)!}$.
Thus 
\begin{equation}\label{Alt}
(\Psi,\Phi^*)=C_n^{-1}Alt\circ\pi\circ (\phi\otimes \psi^*)\circ I(1).
\end{equation}
 We need therefore to apply $Alt$ to the right hand side of \eqref{supfi}.
For shortness set
$$
f(\sigma,\tau,\lambda):=\phi_{\sigma^{-1}\circ\tau\circ\lambda}(B_{\lambda(1)}A_1 ,B_{\lambda(2)}A_2 ,\ldots,B_{\lambda(n-1)}A_{n-1},B_{\lambda(n)} )
$$

Let us apply the procedure of alternation to a term  $f(\sigma,\tau,\lambda)$.  If $\sigma^{-1}\circ\tau\circ\lambda$ is not a full cycle, then $Alt(f(\sigma,\tau,\lambda))=0$. To check this we need only to find an odd permutation in $S_{2n-1}$ that fixes the term $f(\sigma,\tau,\lambda)$. Let $c_1\cdots c_s$ be the cycle decomposition of $\sigma^{-1}\circ\tau\circ\lambda$. We can assume that $c_1=(i_1\cdots i_k)$ is a cycle that does not contain $\lambda(n)$. It follows that, if  $M_2,\ldots,M_s$ are the products of matrices corresponding to  cycles $c_2,\ldots,c_s$,
\begin{align*}f(\sigma,\tau,\lambda)&=tr(B_{\lambda(i_1)}A_{i_1} B_{\lambda(i_2)}A_{i_2}\cdots B_{\lambda(i_k)}A_{i_k})tr(M_2)\cdots tr(M_s)\\
&=tr(A_{i_1} B_{\lambda(i_2)}A_{i_2}\cdots B_{\lambda(i_k)}A_{i_k}B_{\lambda(i_1)})tr(M_2)\cdots tr(M_s)
\end{align*} 
and the last equality gives an odd permutation (a cycle of length $2k$) in $S_{2n-1}$ that fixes  $f(\sigma,\tau,\lambda)$.
If $\sigma^{-1}\circ\tau\circ\lambda$ is a full cycle $(j_1\cdots j_n)$, we can assume that $j_n=n$. Then
 $$\epsilon_\lambda f(\sigma,\tau,\lambda)=\epsilon_\lambda tr(B_{\lambda(j_1)}A_{j_1} B_{\lambda(j_2)}A_{j_2} \cdots B_{\lambda(j_{n-1})}A_{j_{n-1}}B_{\lambda(n)}).
 $$  
 Let $\mu\in S_{2n-1}$ be defined by $\mu(i)=n+1$ for $i=1,\ldots,n-1$ and $\mu(i)=i-n+1$ for $i=n,\ldots,2n-1$. If $\omega\in S_n$ we can consider $\omega$ as an element of $S_{2n-1}$ (fixing $n+1,\ldots ,2n-1$). Let $\nu\in S_n$ be defined by $\nu(i)=j_i$. Then $(\nu\circ \mu^{-1}\circ\lambda\circ\nu\circ\mu)^{-1}$ is the permutation mapping $tr(B_{\lambda(j_1)}A_{j_1} B_{\lambda(j_2)}A_{j_2} \cdots B_{\lambda(j_{n-1})}A_{j_{n-1}}B_{\lambda(n)})$ to 
$tr(B_{ 1 }A_1 B_{ 2 }A_2 \cdots B_{ n-1 }A_{n-1}B_n )$. Since the sign of $(\nu\circ \mu^{-1}\circ\lambda\circ\nu\circ\mu)^{-1}$ is $\epsilon_\lambda$, we see that, if $\sigma^{-1}\circ\tau\circ\lambda$ is a full cycle,
\begin{equation}\label{conclu}
Alt(\epsilon_\lambda f(\sigma,\tau,\lambda))= \frac{1}{(2n-1)!}tr(St_{2n-1}(B_{ 1 },A_1 ,B_{ 2 },A_2 ,\ldots,B_{ n-1 },A_{n-1}, B_n )).
\end{equation}
We are now ready to prove the key result of this section.
 \begin{theorem}\label{main} $(\Psi,\Phi^*)=\frac{(-1)^{\binom{n}{2}}}{n!}T_{n-1}$.
 \end{theorem}
  \begin{proof}Combining \eqref{supfi}, \eqref{Alt}, and \eqref{conclu} we have
  \begin{equation*}
(\Psi,\Phi^*)= \frac{C}{(n!)^3(n-1)!}tr(St_{2n-1}(B_{ 1 },A_1 ,B_{ 2 },A_2 ,\ldots,B_{ n-1 },A_{n-1}, B_n )).
\end{equation*}
where $C$ is the number of  triples  $\sigma,\tau,\lambda$ such that $\sigma^{-1}\circ\tau\circ\lambda$ is a full cycle. There are $ (n!)^2(n-1)!$ such triples. 
\end{proof} 

\begin{theorem}\label{IlT}
$ Hom_\g(S^n(V ), \bigwedge  \mathfrak g^*\otimes  \bigwedge ^nV ) \cong (S^n(V^*)\otimes \bigwedge  \mathfrak g^*)^{\mathfrak g}$ is a free module on the two generators  $\Phi,\Psi$ over $\bigwedge ( T_1,\ldots,T_{n-2})$.
\end{theorem}
\begin{proof}
We first prove that $\Psi$ and $\delta\Psi$ freely generate $(S^n(V^*)\otimes \bigwedge  \mathfrak g^*)^{\mathfrak g}$ over 
$\bigwedge (T_1,\ldots,T_{n-2})$.
Using the formula \eqref{reeder}  it is enough  to prove that the two elements are linearly independent over $\bigwedge (T_1,\ldots,T_{n-2})$. 

Let $\{e_i\}$ be a basis of weight vectors for $V$ with $e_1$ a highest weight vector. Let $\{E_{ij}\}$ be the basis of $End(V)$ of elementary matrices and $\{E^{ij}\}$ the dual basis. Then it is not hard to check that, up to a constant depending on the choice of a trivialization of $\bigwedge^n V$, we have
$$
\Psi(e_1)=E^{21}\wedge\cdots\wedge E^{n1}.
$$
Since $[E_{i1},E_{j1}]=0$ if $i,j\ne 1$, we see that $\partial(\Psi(e_1))=0$. By equivariance, we obtain that $\partial\Psi=0$. Recall (see \cite[(94)]{K}) that the Laplacian $\delta\partial+\partial\delta$ equals $\frac{1}{2}\sum_{i=1}^{\dim\g}\theta(z_i)^2$, where $\{z_i\}$ is an orthonormal basis of $\g$ with respect to the Killing form and $\theta$ is the extension of $ad$ to $\bigwedge \g$.  It follows that $\partial\delta\Psi=(\delta\partial+\partial\delta)\Psi=c\Psi$ with $c$ a non-zero scalar. 
 We can then argue as in the previous section and deduce that is enough to prove that an identity  $a\wedge \Psi= 0,\ a\in \bigwedge( T_1,\ldots,T_{n-2})$, implies $a=0$. For this,  we compute $(a\wedge\Psi,\Phi^*)$  and have, by Theorem \ref{main}, that $0=(a\wedge \Psi,\Phi^*)=\frac{(-1)^{\binom{n}{2}}}{n!}a\wedge T_{n-1}$. Since the relation $a\wedge T_{n-1}=0$ with $a\in \bigwedge ( T_1,\ldots,T_{n-2})$ implies $a=0$, we have proven that  $\Psi$ and $\delta\Psi$ freely generate $(S^n(V^*)\otimes \bigwedge  \mathfrak g^*)^{\mathfrak g}$ over 
$\bigwedge (T_1,\ldots,T_{n-2})$. 
This in particular proves that $\dim Hom_\g(S^n(V),\bigwedge^n\g)=1$, thus $\delta\Psi$ is a multiple of $\Phi$, hence the proof is complete.
\end{proof}
\section{Small representations in type $A$}
For  $sl(n,\C)$,  one can show
that an highest weight module $V$ is small if and only if the highest weight of either $V$ or $V^*$
comes from
a partition of $n$. This means the following: given a partition $\lambda_1\geq\ldots\geq\lambda_n$ of $n$,  the  corresponding highest weight $\lambda$ is $0$ if $\lambda_1=\cdots=\lambda_n=1$ or  $\lambda=\sum_{i=1}^{n-1} a_i\omega_i$ where $\omega_1,\ldots,\omega_{n-1}$ are the fundamental weights and $a_i$ is the number of columns of length $i$ of the partition. For such weights, Stembridge has proved the following formula (cf. \cite[Corollary 6.2]{S}), yielding the graded multiplicities
$M_\lambda(q)$ of the corresponding modules in $\bigwedge sl(n,\C)$. Display the Young diagram in the English way, label the boxes as matrix entries and denote by $h(i,j)$ the hook length of the box $(i,j)$, i.e. the number of boxes strictly on the right of box $(i,j)$ plus the number of boxes strictly below box $(i,j)$ plus one. Set, as usual, $[n]_q=\frac{1-q^n}{1-q}$ and $[n]_q!=\prod_{i=1}^n[i]_q$. Then
\begin{equation}\label{es}
M_\lambda(q)=\frac{[n]_{q^2}!}{1+q}\prod_{(i,j)\in\lambda}\frac{q^{2i-1}+q^{2j-2}}{1-q^{2h(i,j)}}.
\end{equation}
Notice that, since we are dealing with $sl(n,\C)$ rather than $gl(n,\C)$, there is  an extra factor $1/(1+q)$ in the right hand side of \eqref{es} w.r.t. the formula displayed in \cite{S}.
\begin{proposition} If  $\g=sl(V)$ and $V(\lambda)$ is an irreducible non-trivial representation of $\g$ with $\lambda$ corresponding to a partition of $n=\dim V$, then $Hom_\g(V(\lambda),\bigwedge \g)$ is free over $\bigwedge(P_1,\ldots,P_{n-2})$ if and only if
$V(\lambda)$ is either $S^n(V)$ or the adjoint representation.
\end{proposition}
\begin{proof} The fact that the adjoint representation and $S^n(V)$ have the desired property has been shown in \cite{DPP} and in Section \ref{3} above, respectively.
Assume now that $\lambda$ corresponds to a partition of $n$. We can assume $n\ge 4$: if $n\le 3$ the result is trivially verified. If   $Hom_\g(V(\lambda),\bigwedge \g)$ is free over $\bigwedge(P_1,\ldots,P_{n-2})$, the polynomial affording its graded multiplicities in 
$\bigwedge \g$ has to be divisible by $\prod_{i=1}^{n-2} (1+q^{2i+1})$. Use now formula
\eqref{es}. Look at the highest term $1+q^{2n-3}$ in the graded multiplicities of  $\bigwedge(P_1,\ldots,P_{n-2})$. The only possible simplification occurs in the term  
$\prod_{(i,j)\in\lambda}(q^{2i-1}+q^{2j-2})$ of \eqref{es}. This can happen just in the following three cases:
\begin{enumerate}
\item $i=n,\,j=1$;
\item $i=n-1,\,j=1$;
\item $i=1,\,j=n$.
\end{enumerate}
The first case gives the partition corresponding to the trivial representation, which  is excluded. In the second case, since we are excluding the case where $\lambda_n=1$,  the partition is necessarily $(2,1^{n-2})$, corresponding to the adjoint representation. In the third case the partition is necessarily $(n)$, which corresponds to $S^n(V)$.
\end{proof}

 \end{document}